\documentclass[11pt,fullpage, doublespace]{amsart}
\usepackage{amsfonts, amssymb, amsmath}
\usepackage[all]{xy}
\usepackage{bm}
\usepackage{graphicx}

\newtheorem{theorem}{Theorem}[section]

\newtheorem*{theorem*}{Theorem 1.3'}

\newtheorem{lemma}{Lemma}[section]

\newtheorem{corr}{Corollary}[section]
\newtheorem{conj}{Conjecture}[section]

\newcommand{\Z}{\mathbb{Z}}

\newcommand{\nid}{\noindent}

\title{The average number of
divisors of the Euler function}
\author{Kim, Sungjin}
\begin{document}
    \maketitle
\begin{abstract}
The upper bound and the lower bound of average numbers of divisors of Euler Phi function and Carmichael Lambda function are obtained by Luca and Pomerance (see ~\cite{LP}). We improve the lower bound and provide a heuristic argument which suggests that the upper bound given by ~\cite{LP} is indeed close to the truth.
\end{abstract}
\section{Introduction}\footnote{Keywords: Euler, Carmichael, Number of Divisors, Average, AMS Subject Classification Code: 11A25}
Let $n\geq 1$ be an integer. Denote by $\phi(n)$, $\lambda(n)$, the Euler Phi function and the Carmichael Lambda function, which output the order and the exponent of the group $(\Z/n\Z)^{*}$ respectively. We use $p(\mathrm{or} \ p_i)$, $q(\mathrm{or } \ q_i)$ to denote the prime divisors of $n$ and $\phi(n)$ respectively. Then it is clear that $\lambda(n)|\phi(n)$ and the set of prime divisors $q$ of $\phi(n)$ and that of $\lambda(n)$ are identical. Let $n=p_1^{e_1}\cdots p_r^{e_r}$ be a prime factorization of $n$. Then we can compute $\phi(n)$ and $\lambda(n)$ as follows:
$$
\phi(n)=\prod_{i=1}^r \phi(p_i^{e_i}), \  \mathrm{and} \ \lambda(n)=\mathrm{lcm}\left(\lambda(p_1^{e_1}), \ldots, \lambda(p_r^{e_r})\right)
$$
where $\phi(p_i^{e_i})=p_i^{e_i-1}(p_i-1)$ and $\lambda(p_i^{e_i})=\phi(p_i^{e_i})$ if $p_i>2$ or $p_i=2$ and $e_i=1, 2$, and $\lambda(2^e)=2^{e-2}$ if $e\geq 3$.

From the work of Hardy and Ramanujan ~\cite{HR}, it is well known that the normal order of $\tau(n)$ is $(\log n)^{\log 2 + o(1)}$. On the other hand, the average order $\frac1x\sum\limits_{n\leq x} \tau(n)$ is known to be $\log x + O(1)$ which is somewhat larger than the normal order. For $\tau(\lambda(n))$ and $\tau(\phi(n))$, the normal orders of these follows from ~\cite{EP} that they are $2^{\left(\frac12 +o(1)\right)(\log\log n)^2}$. On the contrary, the work of Luca and Pomerance ~\cite{LP} showed that their average order is significantly larger than the normal order. Define $F(x) = \exp\left( \sqrt{\frac{\log x}{\log\log x}}\right)$. In ~\cite[Theorem 1,2]{LP}, they proved that
$$F(x)^{b_1+o(1)} \le \frac 1x\sum_{n\le x}\tau(\lambda(n))
\le \frac1x\sum_{n\le x}\tau(\phi(n)) \le F(x)^{b_2+o(1)}$$
as $x\to\infty$, where
$b_1 = \frac17 e^{-\gamma/2}$ and $b_2 = 2\sqrt 2 e^{-\gamma/2}$.

In this paper we are able to raise the constant $b_1$ so that
it is almost $b_2$, differing only by a factor $\sqrt{2}$. Here, we take advantage of the inequalities of Bombieri-Vinogradov type regarding primes in arithmetic progression (see ~\cite[Theorem 9]{BFI}, also ~\cite[Theorem 2.1]{F}). In this paper, we apply the following version which can be obtained from ~\cite[Theorem 2.1]{F}:
For $(a,n)=1$, we write $E(x;n,a):=\pi(x;n,a)-\frac{\pi(x)}{\phi(n)}$. Let $0<\lambda<1/10$. Let $R\leq x^{\lambda}$. For some $B=B(A)>0$, $M=\log^B x$, and $Q=x/M$,
$$
\sum_{\substack{{r\leq R}\\{(r,a)=1}}}\left|\sum_{\substack{{q\leq \frac Q{ r}}\\{(q,a)=1}}}E(x;qr,a) \right|\ll_{A,\lambda}   x \log^{-A} x.
$$
In fact, ~\cite[Theorem 2.1]{F} builds on ~\cite[Theorem 9]{BFI} and obtains a more accurate estimate, but we only need the above form for our purpose. Note that one of the important differences between ~\cite[Theorem 9]{BFI} and ~\cite[Theorem 2.1]{F} is the presence of $\frac Qr$ in the inner sum. This will be essential in the proof of our lemmas (see Lemma 2.2 and 2.3).

It is interesting to note that one of these improvements is related to a Poisson distribution that we can obtain from prime numbers. Another point of improvement comes from the idea in the proof of Gauss' Circle Problem.
\begin{theorem}
As $x\rightarrow \infty$, we have
$$
\sum_{n\leq x}\tau(\phi(n))\geq \sum_{n\leq x} \tau(\lambda(n))\geq x\exp\left(2 e^{-\frac{\gamma}2}\sqrt{\frac{\log x}{\log\log x}}(1+o(1))\right).
$$
\end{theorem}
It is clear from $\lambda(n)|\phi(n)$ that $\sum_{n\leq x}\tau(\lambda(n))\leq \sum_{n\leq x}\tau(\phi(n))$. A natural question to ask is how large is the latter compared to the former. Luca and Pomerance proved in ~\cite[Theorem 2]{LP} that
$$
\frac1x \sum_{n\leq x}\tau(\lambda(n))=o\left( \max_{y\leq x} \frac1y \sum_{n\leq y}\tau(\phi(n)) \right).
$$
Moreover, they mentioned that a stronger statement
$$
\frac1x\sum_{n\leq x}\tau(\lambda(n)) = o \left(\frac1x\sum_{n\leq x}\tau(\phi(n))\right)$$
is probably true, but they did not have the proof.
Here, we prove that this statement is indeed true. As in the proof of  ~\cite[Theorem 2]{LP}, we take advantage of the fact that prime $2$ appears rarely in the factorization of $\lambda(n)$ than in the factorization of $\phi(n)$.
\begin{theorem}
As $x\rightarrow\infty$, we have
$$
\sum_{n\leq x}\tau(\lambda(n)) = o \left(\sum_{n\leq x}\tau(\phi(n))\right).$$
\end{theorem}
Finally, we give a heuristic argument suggests that the constant in the upper bound is indeed optimal. Here, we try to extend the method in the proof of Theorem 1.1 by devising a binomial distribution model. However, we were unable to prove it. The main difficulty is due to the short range of $u$ ($u<\log^{A_1} x$) in the lemmas (see Lemma 2.1, 2.3, Corollary 2.1, and 2.2).
\begin{conj}
As $x\rightarrow\infty$, we have
$$
\sum_{n\leq x} \tau(\lambda(n))=x\exp\left( 2\sqrt 2 e^{-\frac{\gamma}2} \sqrt{\frac{\log x}{\log\log x}}(1+o(1))\right).
$$
\end{conj}
Throughout this paper, $x$ is a positive real number, $n$, $k$ are positive integers, and $p$, $q$ are prime numbers. We use Landau symbols $O$ and $o$. Also, we write $f(x)\asymp g(x)$ for positive functions $f$ and $g$, if $f(x)=O(g(x))$ and $g(x)=O(f(x))$. We will also use Vinogradov symbols $\ll$ and $\gg$. We write the iterated logarithms as $\log_2 x = \log\log x$ and $\log_3 x = \log\log\log x$. The notations $(a,b)$ and $[a,b]$ mean the greatest common divisor and the least common multiple of $a$ and $b$ respectively. We write $P_z=\prod_{p\leq z} p$. We also use the following restricted divisor functions:
$$
\tau_z(n):=\prod_{\substack{{p^e||n}\\{p>z}}}\tau(p^e), \ \
\tau_{z,w}(n):=\prod_{\substack{{p^e||n}\\{z<p\leq w}}}\tau(p^e), \ \
\textrm{and } \
\tau_z'(n):=\prod_{\substack{{p^e||n}\\{p\leq z}}}\tau(p^e).$$
Moreover, for $n>1$, denote by $p(n)$ the smallest prime factor of $n$.

\bf Acknowledgement. \rm
The author would like to thank Carl Pomerance for encouraging him to work on this problem, and numerous valuable comments and conversations.

\section{Lemmas}
The following lemma is ~\cite[Lemma3]{LP} with a slightly relaxed $z$, and it is essential toward proving the theorem. This is stated and proved with the Chebyshev functions $\psi(x):=\sum\limits_{n\leq x} \Lambda(n)$  and $\psi(x;q,a):=\sum\limits_{n\leq x, \ n\equiv a \ \mathrm{mod} \ q} \Lambda(n)$ in ~\cite{LP2}. Here, we use the prime counting functions $\pi(x):=\sum\limits_{p\leq x} 1$  and $\pi(x;q,a):=\sum\limits_{p\leq x, \ p\equiv a \ \mathrm{mod} \ q} 1$ instead. We are allowed to do these replacements by applying the partial summation.
\begin{lemma}
Let $0<\lambda<\frac1{10}$. Assume that  $z\leq \lambda \log x$. Then for any $A>0$, there is $B=B(A)>0$ such that for $M=\log^B x$, and $Q=\frac xM$,
\begin{equation}
E_z(x):=\sum_{r|P_z}\mu(r)\sum_{\substack{{n\leq Q}\\{r|n}}}\left(\pi(x;n,1)-\frac{\pi(x)}{\phi(n)}\right)\ll_{A,  \lambda}\frac x{\log^A x}. \end{equation}

Let $0<\lambda <\frac1{10}$. Assume that $u$ is a  positive integer with $p(u)>z$, $u<(\log x)^{A_1}$  and $\tau(u)<A_1$. Then for any $A>0$, there is $B=B(A,A_1)>0$ such that for $M=\log^B x$, and $Q=\frac xM$,
\begin{equation}
E_{u,z}(x):=\sum_{r|P_z}\mu(r)\sum_{\substack{{n\leq Q}\\{r|n}}}\left(\pi(x;[u,n],1)-\frac{\pi(x)}{\phi([u,n])}\right)\ll_{A, A_1,\lambda} \frac{x}{\log^A x}.
\end{equation}
\end{lemma}
 \begin{proof}[Proof of (1)]
For $(a,n)=1$, we write $E(x;n,a):=\pi(x;n,a)-\frac{\pi(x)}{\phi(n)}$. If $r|P_z$, we have by the Prime Number Theorem, $r\leq R:=P_z=\exp(z+o(z))\leq x^{\lambda'}$ with $0<\lambda'<1/10$. By partial summation and diadically applying ~\cite[Theorem 2.1]{F}, we have for $B=B(A)>0$, $M=\log^B x$, and $Q=x/M$,
\begin{equation}
\sum_{\substack{{r\leq R}\\{(r,a)=1}}}\left|\sum_{\substack{{q\leq \frac Q{ r}}\\{(q,a)=1}}}E(x;qr,a) \right|\ll_{A,\lambda} \frac x{\log^A x}.
\end{equation}
Taking $a=1$ and   $|\mu(r)|\leq 1$, (1) follows.
\end{proof}
\begin{proof}[Proof of (2)]
Let $d\leq x^{\epsilon}$ so that $dR\leq x^{\lambda'}$ with $0<\lambda'<1/10$. By (3), there exist $B=B(A)>0$ such that we have for $M=\log^B x$ and $Q=x/M$,
\begin{equation}
\sum_{ r\leq R }\left|\sum_{ q\leq \frac Q{ r} }E(x;dqr,1)\right|=\sum_{\substack{{r\leq dR}\\{r\equiv 0 \ \mathrm{ mod } \ d} }}\left|\sum_{ q\leq \frac Q{  r} }E(x;qr,1)\right|
\leq \sum_{ r\leq dR}  \left|\sum_{ q\leq \frac Q{   r}} E(x;qr,1)\right|\ll_{A,\lambda} \frac x{\log^A x}.
\end{equation}
By $(u,r)=1$, we have $[u,n]=[u,qr]=r[u,q]=ruq/(u,q)$. We partition the set of $q\leq \frac Qr$ as $\bigcup_{d|u} A_d$, where $q\in A_d$ if and only if $(u,q)=d$. Let $B_{Q,d}=\left\{q\leq \frac Qr: q \equiv 0 \ \mathrm{ mod } \ d\right\}$. By inclusion-exclusion, we have for any $d|u$,
$$
\sum_{q\in A_d}E\left(x;\frac{ruq}d,1\right) = \sum_{s|\frac ud} \mu(s) \sum_{q\in B_{Q,ds}}E\left(x;\frac{ruq}d,1\right).
$$
It is clear that
$$\sum_{q\in B_{Q,ds}}E\left(x;\frac{ruq}d,1\right)=\sum_{q\in B_{\frac{uQ}d,us}}E(x;qr,1).$$
Since $r\leq R:=P_z<x^{\lambda'}$ with $\lambda'<\frac1{10}$, $\frac{uQ}d\leq Q\log^{A_1}x$, and $us<\log^{2A_1} x < x^{\epsilon}$, we have by (4),
$$\sum_{r\leq R} \left|\sum_{q\in B_{\frac{uQ}d,us}}E(x;qr,1)\right|\ll_{A,  A_1, \lambda}   \frac x {\log^{A} x} $$
with a suitable choice of $B=B(A, A_1)$.
Then
\begin{align*}\sum_{r\leq R}\left|\sum_{q\in A_d}E\left(x;\frac{ruq}d,1\right)\right| &=\sum_{r\leq R}\left|\sum_{s|\frac ud}\mu(s)\sum_{q\in B_{Q,ds}}E\left(x;\frac{ruq}d,1\right)\right|\\
&\leq \sum_{s|\frac ud}\sum_{r\leq R}\left|\sum_{q\in B_{Q,ds}}E\left(x;\frac{ruq}d,1\right)\right|\\
&\ll_{A,A_1,\lambda} \tau\left(\frac ud\right) \frac x {\log^{A} x}.\end{align*}
Thus, summing over $d|u$, we have
\begin{align*}\left|\sum_{r|P_z}\mu(r) \sum_{q\leq \frac Qr} E(x;[u,qr],1)  \right| &\leq \sum_{d|u}\sum_{r\leq R}\left|\sum_{q\in A_d}E\left(x;\frac{ruq}d,1\right)\right|\\
&\ll_{A,A_1,\lambda} (\tau(u))^2 \frac x {\log^{A} x}\ll_{A, A_1,\lambda}  \frac x {\log^{A} x}.\end{align*}
Thus, we have the result (2). \end{proof}
The following is ~\cite[Lemma 5]{LP} with a slightly relaxed $z$.
\begin{lemma}
Let $0<\lambda<\frac1{10}$, and $1<z\leq \lambda \log x$.  Let $c_1=e^{-\gamma}$. Then we have
\begin{equation}
R_z(x):=\sum_{p\leq x} \tau_z(p-1)=c_1 \frac x{\log z} + O\left(\frac x{\log^2 z}\right),
\end{equation}
and for $1<z\leq \frac{\log x}{\log_2^2 x}$,
\begin{equation}
S_z(x):=\sum_{p\leq x}\frac{\tau_z(p-1)}p=c_1\frac{\log x}{\log z} + O\left(\frac{\log x}{\log^2 z}\right).
\end{equation}
\end{lemma}
\begin{proof}[Proof of (5)]
Take $A=2$ and the corresponding $B(A)$ and $M$ in Lemma 2.1(1). Then by inclusion-exclusion,
$$
R_z(x)= \sum_{d\in D_z(x)} \pi(x;d,1)= \sum_{d\in D_z\left(\frac xM\right)} \pi(x;d,1)+\sum_{r|P_z}\mu(r)\sum_{\frac x{rM}<q\leq \frac xr}\pi(x;qr,1)=R_1+R_2, \textrm{ say.}$$
By ~\cite[Lemma 4]{LP} and Lemma 2.1(1),
$$
R_1=\sum_{d\in D_z\left(\frac xM\right)}  \frac{\pi(x)}{\phi(d)} + \sum_{r|P_z}\mu(r)\sum_{q\leq \frac x{rM}} E(x;qr,1) =c_1\frac x{\log z}+O\left(\frac x{\log^2 z}\right) + O\left(\frac x{\log^2 x}\right).$$
By divisor-switching technique and Brun-Titchmarsh inequality as in ~\cite{LP2}, we have
$$
R_2\ll \sum_{r|P_z}\sum_{k\leq M}\pi(x;rk,1)\ll \sum_{r|P_z}\sum_{k\leq M} \frac{x}{\phi(rk) \log x}\ll \frac{x\log z\log M}{\log x}\ll \frac x{\log^2 z}.$$
Therefore, (5) follows.
 \end{proof}
\begin{proof}[Proof of (6)]
By partial summation,
$$S_z(x)=\frac{R_z(t)}t|_2^x + \int_2^x \frac{R_z(t)}{t^2}dt.$$
We split the integral at $z =  \lambda\log t$. Then by (4),
$$
\int_{z\leq \lambda \log t} \frac{R_z(t)}{t^2}dt = \int_{e^{z/\lambda}}^x \left(c_1 \frac t{\log z} + O\left(\frac t{\log^2 z}\right)\right)\frac{dt}{t^2}=c_1\frac{\log x}{\log z} + O\left(\frac{\log x}{\log^2 z}\right).$$
On the other hand, by the trivial bound $R_z(t)\ll t$,
$$
\int_{z>\lambda \log t} \frac{R_z(t)}{t^2}dt \ll \int_2^{e^{z/\lambda}} t \frac{dt}{t^2} \ll  z.$$
Since $z\log^2 z\ll \log x$, (6) follows.
\end{proof}
The following is ~\cite[Lemma 6]{LP} with a wider range of $z$. This relaxes the rather severe restriction $z\leq \frac{\sqrt{\log x}}{\log_2^6 x}$.
\begin{lemma}
Let $1\leq u \leq x$ be any positive integer. Then
\begin{equation}
R_{u,z}(x):=\sum_{\substack{{p\leq x}\\{p \equiv 1 \ \mathrm{ mod } \ u}}}\tau_z(p-1) \ll \frac{\tau(u)}{\phi(u)}x, \ \ S_{u,z}(x):=\sum_{\substack{{p\leq x}\\{p \equiv 1 \ \mathrm{ mod } \ u}}}\frac{\tau_z(p-1)}p\ll\frac{\tau(u)}{\phi(u)}\log x,
\end{equation}
and $\phi(u)$ can be replaced by $u$ if $p(u)>z$ and $\tau(u)<A_1$.

Assume that $u$ is a positive integer with $p(u)>z$, $u< (\log x)^{A_1}$ and $\tau(u)<A_1$. Then for $z\leq \lambda \log x$,
\begin{equation}
R_{u,z}(x) =\frac{\tau(u)}{u}R_z(x) \left(1+O\left(\frac 1{\log z}\right)\right),
\end{equation}
and for $z\leq \frac{\log x}{\log_2^2 x}$,
\begin{equation}
S_{u,z}(x) =\frac{\tau(u)}u S_z(x)\left(1+O\left(\frac 1{\log z}\right)\right).
\end{equation}
\end{lemma}
\begin{proof}[Proof of (7)]
This is a uniform version of ~\cite[Lemma 3.7]{Pe}. We apply Dirichlet's hyperbola method as it was done in ~\cite[Lemma 3.7]{Pe}. First, we see that
$$
R_{u,z}(x)\leq \sum_{\substack{{p\leq x}\\{p \equiv 1 \ \mathrm{ mod } \ u}}}\tau(p-1)
 \leq \sum_{\substack{{p\leq x}\\{p \equiv 1 \ \mathrm{ mod } \ u}}}\tau\left(\frac{p-1}u\right)\tau(u)
 \leq 2\tau(u) \sum_{k\leq \sqrt{\frac xu}} \pi(x;ku,1).
$$
Since the sum is zero for $x\leq u$, we may assume that $x>u$. By Brun-Titchmarsh inequality,
$$
\pi(x;ku,1)\leq \frac{2x}{\phi(ku)\log\left(\frac{x}{ku}\right)}\leq \frac{4x}{\phi(u)\phi(k)\log\frac xu}.$$
Thus, summing over $k$ gives
$$
\sum_{k\leq \sqrt{\frac xu}} \pi(x;ku,1)\leq \frac{8x}{\phi(u)}\sum_{d=1}^{\infty}\frac{\mu^2(d)}{d\phi(d)}.
$$
Therefore, we have the result. The estimate for $S_{u,z}$ follows from partial summation.

We remark that for $u$ with $p(u)>z$,
$$
\frac{u\phi(d)}{\phi(ud)}=\prod_{p|u, p\nmid d} \left(1-\frac1p\right)^{-1}=1+O\left(\frac{\tau(u)}z\right), \ \ \
\frac1{\phi(u)}=\frac1u \prod_{p|u}\left(1-\frac1p\right)^{-1}=\frac1u \left(1+O\left(\frac{\tau(u)}z\right)\right).
$$
Therefore, $\phi(u)$ can be replaced by $u$ if $p(u)>z$ and $\tau(u)<A_1$.
\end{proof}
\begin{proof}[Proof of (8)]
We begin with
$$
R_{u,z}(x)=\sum_{d\in D_z(x)}\pi(x;[u,d],1).$$
Let $A>0$ be a positive number that $\frac x{\log^A x} \ll \frac {\tau(u)}{u} \frac{x}{\log^2 x}$, and $B(A)$ and $M$ be the corresponding parameters depending on $A$ in Lemma 2.1(2). By inclusion-exclusion,
$$
\sum_{d\in D_z(x)}\pi(x;[u,d],1)=\sum_{d\in D_z\left(\frac xM\right)}\pi(x;[u,d],1)+\sum_{r|P_z}\mu(r)\sum_{\frac{x}{rM}<q\leq \frac xr}\pi(x;[u,qr],1)=R_1+R_2, \ \textrm{ say}.$$
By Lemma 2.1(2), we have
$$
R_1=\sum_{d\in D_z\left(\frac xM\right)} \frac{\pi(x)}{\phi([u,d])}+\sum_{r|P_z}\mu(r)\sum_{q\leq \frac x{rM}} E(x;[u,qr],1)=\sum_{d\in D_z\left(\frac xM\right)} \frac{\pi(x)}{\phi([u,d])}+O\left(\frac {\tau(u)}{u} \frac{x}{\log^2 x}\right).
$$
The first sum is treated as follows:
\begin{align*}
\sum_{d\in D_z\left(\frac xM\right)} \frac{\pi(x)}{\phi([u,d])}&=\sum_{d_1\in D_z\left(\frac x{uM}\right)} \frac{\pi(x)N_{d_1}}{\phi(ud_1)}+O\left(\pi(x)\sum_{\substack{{\frac x{uM}<d_1\leq \frac xM }\\{p(d_1)>z}}}\frac {\tau(u)}{\phi(ud_1)} \right)\\
&=\sum_{d_1\in D_z\left(\frac x{uM}\right)} \frac{\pi(x)N_{d_1}}{\phi(ud_1)}+O\left(\pi(x)\frac{\tau(u) \log u }{\phi(u)\log z}\right)\\
&= \sum_{d_1\in D_z\left(\frac x{uM}\right)} \frac{\pi(x)N_{d_1}}{\phi(ud_1)}+O\left(\frac{\tau(u)}u \frac x{\log^2 z}\right),
\end{align*}
where $N_{d_1}=\left|\{d\in D_z\left(\frac xM\right) : [u, d]=ud_1\}  \right|$. Since $N_{d_1}\leq \tau(u)$ and $\phi(ud_1)\geq \phi(u)\phi(d_1)$, by ~\cite[Lemma 4]{LP},
$$
\sum_{d_1\in D_z\left(\frac x{uM}\right)} \frac{\pi(x)N_{d_1}}{\phi(ud_1)}\leq \frac{\tau(u)}{\phi(u)}\left(c_1\frac{x}{\log z}+O\left(\frac{x}{\log^2 z}\right)\right).
$$
Thus, we have the upper bound
$$
\sum_{d_1\in D_z\left(\frac x{uM}\right)} \frac{\pi(x)N_{d_1}}{\phi(ud_1)}\leq \frac{\tau(u)}{u}\left(c_1\frac{x}{\log z}+O\left(\frac{x}{\log^2 z}\right)\right).$$
On the other hand, $N_{d_1}=\tau(u)$ if $(u,d_1)=1$. Then, we may apply ~\cite[Lemma 4]{LP} since $P(u)\leq\log^{A_1} x$, we obtain that
\begin{align*}
\sum_{d_1\in D_z\left(\frac x{uM}\right)} \frac{\pi(x)N_{d_1}}{\phi(ud_1)}&\geq \frac{\tau(u)}u \left( \sum_{\substack{{d_1\in D_z\left(\frac x{uM}\right)}\\{(u,d_1)=1}}} \frac{\pi(x)}{\phi(d_1)}+O\left(\frac x{\log^2 z}\right)\right)\\&\geq \frac{\tau(u)}u \frac{\phi(u)}u \left(c_1\frac{x}{\log z}+O\left(\frac{x}{\log^2 z}\right)\right).\end{align*}
Thus, we have the lower bound
$$
\sum_{d_1\in D_z\left(\frac x{uM}\right)} \frac{\pi(x)N_{d_1}}{\phi(ud_1)}\geq\frac{\tau(u)}{u}\left(c_1\frac{x}{\log z}+O\left(\frac{x}{\log^2 z}\right)\right).
$$
This shows that
$$
R_1=\frac{\tau(u)}{u}\left(c_1\frac{x}{\log z}+O\left(\frac{x}{\log^2 z}\right)\right).$$
By divisor-switching technique and Brun-Titchmarsh inequality as in ~\cite{LP2}, we have
\begin{align*}
R_2 &\ll \sum_{r|P_z} \sum_{d|u}\sum_{s|\frac ud}\sum_{\substack{{\frac{x}{rM}<q\leq \frac xr}\\{ds|q}}}\pi\left(x;\frac{uqr}d,1\right)\\
&\ll \sum_{r|P_z} \sum_{d|u}\sum_{s|\frac ud}\sum_{\frac x{dsrM}<q\leq \frac x{dsr}}\pi\left(x;rusq,1\right)\\
&\ll \sum_{r|P_z}\sum_{d|u}\sum_{s|\frac ud} \sum_{k\leq \frac{dM}u} \pi(x;rusk ,1)\\
&\ll \sum_{r|P_z}\sum_{d|u}\sum_{s|\frac ud} \sum_{k\leq \frac{dM}u}\frac x{\phi(rusk) \log x}\ll \tau(u)\frac{x\log z\log u \log M}{\phi(u)\log x}\ll \frac{\tau(u)}u \frac x{\log^2 z}.
\end{align*}
This completes the proof of (8).
\end{proof}
\begin{proof}[Proof of (9)] We use (7) and (8), and apply partial summation as in (6).
\end{proof}
The following is used with inequality in ~\cite[Lemma 7]{LP}. Here, we obtain an equality that will be used frequently in this paper.
\begin{lemma}
Let $0<\lambda<\frac1{10}$. Fix $a>1$ and an integer $0\leq B<\infty$. We use $z=\lambda\log x$ for the formula for $R_B$ and $z=\frac{\log x}{\log_2^2  x}$ for the formula for $S_B$. Let $I_a(x)=[z,z^a]$. Define
$$\mathcal{U}_B = \{ u : u \textrm{ is a positive square-free integer consisted of exactly $B$ prime divisors in }I_a(x)\}.$$
Then we have
$$R_B:=\sum_{u\in \mathcal{U}_B} R_{u,z}(x)=\frac{(2\log a)^B}{B!}R_z(x)\left(1+O\left(\frac1{\log z}\right)\right),$$
and
$$S_B:=\sum_{u\in \mathcal{U}_B} S_{u,z}(x)=\frac{(2\log a)^B}{B!}S_z(x)\left(1+O\left(\frac1{\log z}\right)\right).$$
\end{lemma}
\begin{proof}
We apply Lemma 2.3 with $u\in \mathcal{U}_B$. Note that $u\in \mathcal{U}_B$ satisfies the conditions for $u$ in Lemma 2.3(8), (9). Then,
\begin{align*}
\sum_{u\in \mathcal{U}_B} R_{u,z}(x)&=\sum_{u\in\mathcal{U}_B} \frac{\tau(u)}u R_z(x)\left(1+O\left(\frac1{\log z}\right)\right)\\
&= \left( \frac1{B!}\left(\sum_{p\in I_a(x)} \frac2p\right)^B+O\left(\frac1{(B-2)!}\left(\sum_{p\in I_a(x)}\frac4{p^2}\right)\left(\sum_{p\in I_a(x)} \frac2p\right)^{B-2}\right)\right)R_z(x)\left(1+O\left(\frac1{\log z}\right)\right)\\
&=\left( \frac1{B!}\left(\sum_{p\in I_a(x)} \frac2p\right)^B+O\left(\frac1z\right)\right)R_z(x)\left(1+O\left(\frac1{\log z}\right)\right)\\
&=\frac{2^B}{B!}\left(\log\log z^a - \log\log z + O\left(\frac1{\log z}\right)\right)^BR_z(x)\left(1+O\left(\frac1{\log z}\right)\right)\\
&=\frac{(2\log a)^B}{B!}R_z(x)\left(1+O\left(\frac1{\log z}\right)\right).\end{align*}
The result for $S_B$ can be obtained similarly.
\end{proof}
Although we relaxed $z\leq \frac{\sqrt{\log x}}{\log_2^6 x}$ to $z\leq \frac{\log x}{\log_2^2 x}$, the range is still not enough for further use. We will see how this range can be relaxed to $\log^{\frac1A} x<z\leq \log^A x$ in Lemma 2.5. A probability mass function of a Poisson distribution comes up as certain densities.
\begin{lemma}
Let $0<\lambda<\frac1{10}$. Fix $a>1$ and an integer $0\leq B<\infty$. We use $z=\lambda\log x$ for the formula for $R'_B$ and $z=\frac{\log x}{\log_2^2  x}$ for the formula for $S'_B$. Let $I_a(x)=(z,z^a]$. Define
$$\tau_{z,z^a} (n) = \prod_{\substack{{p^e||n}\\{p\in I_a(x)}}} \tau(p^e), \ \ w_{z,z^a}(n)=|\{p|n \  : \ p\in I_a(x)\}|,$$
and
$$
R'_B: =\sum_{\substack{{p\leq x}\\{w_{z,z^a}(p-1)=B}}} \tau_z(p-1), \ \ S'_B: =\sum_{\substack{{p\leq x}\\{w_{z,z^a}(p-1)=B}}} \frac{\tau_z(p-1)}p.$$
Then as $x\rightarrow\infty$, we have
\begin{equation}
R'_B=\frac{(2\log a)^B}{B!a^2}   R_z(x)(1+o(1)), \ \ S'_B=\frac{(2\log a)^B}{B!a^2}  S_z(x)(1+o(1)),
\end{equation}
and we have
\begin{equation}
R_{z^a}(x)=\frac1a R_z(x) (1+o(1)), \ \ S_{z^a}(x)=\frac1a S_z(x) (1+o(1)).
\end{equation}
\end{lemma}
\begin{proof}[Proof of (10)]
We remark that by (7), (8), (9), the contribution of primes $p$ such that $p-1$ is divisible by a square of a prime $q>z$ is negligible. In fact, those contributions to $R_z(x)$ and $S_z(x)$ are $O(R_z(x)/z)$ and $O(S_z(x)/z)$ respectively. Thus, we assume that $p-1$ is not divisible by square of any prime $q>z$. By Lemma 2.4 and inclusion-exclusion principle,
$$
R'_B= R_B - \binom {B+1}1 R_{B+1} + \binom {B+2}2 R_{B+2} - \binom {B+3}3 R_{B+3} +-\cdots.$$
Moreover, for any $k\geq 1$,
$$
\sum_{j=0}^{2k-1} (-1)^j \binom {B+j}j R_{B+j} \leq R'_B \leq \sum_{j=0}^{2k} (-1)^j\binom{B+j}jR_{B+j}.$$
Then dividing by $R_z(x)$ gives
$$
\sum_{j=0}^{2k-1} (-1)^j \binom {B+j}j \frac{R_{B+j}}{R_z(x)} \leq \frac{R'_B}{R_z(x)} \leq \sum_{j=0}^{2k} (-1)^j\binom{B+j}j\frac{R_{B+j}}{R_z(x)}.$$
By Lemma 2.4, we have
$$
\frac{(2\log a)^B}{B!}\sum_{j=0}^{2k-1} (-1)^j \frac{(2\log a)^j}{j!} \left(1+O\left(\frac1{\log z}\right)\right)\leq \frac{R'_B}{R_z(x)}\leq \frac{(2\log a)^B}{B!}\sum_{j=0}^{2k} (-1)^j \frac{(2\log a)^j}{j!} \left(1+O\left(\frac1{\log z}\right)\right).$$
Taking $x\rightarrow\infty$, we have
$$
\frac{(2\log a)^B}{B!}\sum_{j=0}^{2k-1} (-1)^j \frac{(2\log a)^j}{j!} \leq \liminf_{x\rightarrow\infty} \frac{R'_B}{R_z(x)}\leq \limsup_{x\rightarrow\infty}  \frac{R'_B}{R_z(x)}\leq\frac{(2\log a)^B}{B!}\sum_{j=0}^{2k } (-1)^j \frac{(2\log a)^j}{j!}.$$
Letting $k\rightarrow\infty$, we obtain
$$\lim_{x\rightarrow\infty}\frac{R'_B}{R_z(x)}=\frac{(2\log a)^B}{B!a^2}.$$

The result for $S'_B$ can be obtained similarly.
\end{proof}
\begin{proof}[Proof of (11)]
As in the proof of (10), we assume that $p-1$ is not divisible by square of any prime $q>z$.
Note that $\tau_z(p-1)=\tau_{z^a}(p-1)\tau_{z,z^a}(p-1)$. Let $0\leq B<\infty$ be a fixed integer. If $w_{z,z^a}(p-1)=B$ then $\tau_{z,z^a}(p-1)=2^B$. Then we have by (10),
$$
\sum_{\substack{{p\leq x}\\{w_{z,z^a}(p-1)=B}}} \tau_{z^a}(p-1) =\sum_{\substack{{p\leq x}\\{w_{z,z^a}(p-1)=B}}} \frac{\tau_z(p-1)}{2^B}=\frac{R'_B}{2^B}=\frac{(\log a)^B}{B!a^2}   R_z(x)(1+o(1))  .$$
Then by Lemma 2.4,
\begin{align*}
\frac{R_{z^a}(x)}{R_z(x)} &= \sum_{j< B} \frac{(\log a)^j}{j!a^2} (1+o(1))+ \frac1{R_z(x)}\sum_{ j\geq B} \frac1{2^j}\sum_{\substack{{p\leq x}\\{w_{z,z^a}(p-1)=j}}} \tau_{z }(p-1)\\
&=\sum_{j<B} \frac{(\log a)^j}{j!a^2} (1+o(1))+O\left(\frac1{2^BR_z(x)} \sum_{\substack{{p\leq x}\\{w_{z,z^a}(p-1)\geq B}}} \tau_{z }(p-1)  \right)\\
&=\sum_{j< B} \frac{(\log a)^j}{j!a^2} (1+o(1))+O\left(\frac{R_B}{2^BR_z(x)}  \right)\\
&=\sum_{j< B} \frac{(\log a)^j}{j!a^2} (1+o(1))+O\left(  \frac{(2\log a)^B}{2^BB!}\left(1+O\left(\frac1{\log z}\right)\right)\right).
\end{align*}
Thus, both $\liminf\limits_{x\rightarrow\infty} \frac{R_{z^a}(x)}{R_z(x)}$ and $\limsup\limits_{x\rightarrow\infty} \frac{R_{z^a}(x)}{R_z(x)}$ are
$$
\sum_{j\leq B} \frac{(\log a)^j}{j!a^2} + O\left(\frac{(\log a)^B}{B!}\right)
$$
and the constant implied in $O$ does not depend on $B$.
Therefore, letting $B\rightarrow \infty$, we obtain
$$
\lim_{x\rightarrow \infty} \frac{R_{z^a}(x)}{R_z(x)} = \frac1a.
$$
The result for $S_{z^a}(x)$ can be obtained similarly.
\end{proof}
Lemma 2.5 allows us to have an extended range of $z$, and the same method applied to $R_{u,z}(x)$, we can also extend range of $z$ for $R_{u,z}(x)$ and $S_{u,z}(x)$.
\begin{corr}
Fix any $A>1$. Let $\log^{\frac1A} x<z\leq \log^A x$. Then as $x\rightarrow\infty$, we have
\begin{equation}
R_z(x) = c_1 \frac{x}{\log z} (1+o(1)), \ \ S_z(x) = c_1 \frac{\log x}{\log z} (1+o(1)).
\end{equation}
Assume that $u$ is a positive integer with $p(u)>z$, $u< (\log x)^{A_1}$ and $\tau(u)<A_1$. Then as $x\rightarrow\infty$, we have
\begin{equation}
R_{u,z}(x) = \frac{\tau(u)}u R_z(x)(1+o(1)), \ \ S_{u,z}(x) = \frac{\tau(u)}u S_z(x)(1+o(1)).
\end{equation}
\end{corr}
We apply Corollary 2.1 to obtain the following uniform distribution result:
\begin{corr}
Let $2\leq v\leq x$ and $r:={(v^{\frac32}\log v)}^{-1}$.
Suppose also that $r\geq \log^{-\frac45} x$, $0\leq \alpha \leq \beta\leq 1$, and $\beta-\alpha\geq r$. Then for $z\leq\frac{\log x^r}{\log_2^2 x^r}$,
\begin{equation}
\sum_{\alpha \leq \frac{\log p}{\log x} < \beta} \frac{\tau_z(p-1)}p=(\beta-\alpha)S_z(x)\left(1+O\left(\frac1{\log z}\right)\right).
\end{equation}
For $\log^{\frac1A} x<z\leq \log^A x$, we have as $x\rightarrow\infty$,
\begin{equation}
\sum_{\alpha \leq \frac{\log p}{\log x} < \beta} \frac{\tau_z(p-1)}p=(\beta-\alpha)S_z(x)\left(1+o(1)\right).
\end{equation}
Assume that $u$ is a positive integer with $p(u)>z$, $u< (\log x)^{A_1}$ and $\tau(u)<A_1$. Then we have for $z\leq\frac{\log x^r}{\log_2^2 x^r}$,
\begin{equation}
\sum_{\substack{{\alpha \leq \frac{\log p}{\log x} < \beta} \\{p\equiv 1\ \mathrm{ mod } \ u}}} \frac{\tau_z(p-1)}p=(\beta-\alpha)\frac{\tau(u)}uS_z(x)\left(1+O\left(\frac1{\log z}\right)\right).
\end{equation}
and for $\log^{\frac1A} x<z\leq \log^A x$, we have as $x\rightarrow\infty$,
\begin{equation}
\sum_{\substack{{\alpha \leq \frac{\log p}{\log x} < \beta} \\{p\equiv 1\ \mathrm{ mod } \ u}}} \frac{\tau_z(p-1)}p=(\beta-\alpha)\frac{\tau(u)}uS_z(x)\left(1+o(1)\right).
\end{equation}

\end{corr}
\begin{proof}
By Lemma 2.2(5) and partial summation, we have for $\beta-\alpha\geq r$,
\begin{align*}
\sum_{\alpha \leq \frac{\log p}{\log x} < \beta} \frac{\tau_z(p-1)}p&=\frac{R_z(t)}t|_{x^{\alpha}}^{x^{\beta}}+\int_{x^{\alpha}}^{x^{\beta}} \frac{R_z(t)}{t^2}dt\\
&=c_1(\beta-\alpha)\frac{\log x}{\log z}\left(1+O\left(\frac1{\log z}\right)\right) +O\left(\frac1{\log^2 z}\right).
\end{align*}
Clearly, $r\log x \gg 1$. Thus, the second $O$-term can be included in the first $O$-term. Then (14) follows.

Since $r\log x \geq \log^{\frac15} x$, the range $\log^{\frac1A} x<z\leq \log^A x$ can be obtained from taking powers of $\frac{\log x^r}{\log_2^2 x^r}$. We have by (12), as $x\rightarrow\infty$,
\begin{align*}
\sum_{\alpha \leq \frac{\log p}{\log x} < \beta} \frac{\tau_z(p-1)}p&=\frac{R_z(t)}t|_{x^{\alpha}}^{x^{\beta}}+\int_{x^{\alpha}}^{x^{\beta}} \frac{R_z(t)}{t^2}dt\\
&=c_1(\beta-\alpha)\frac{\log x}{\log z}\left(1+o(1)\right) +o\left(\frac1{\log z}\right).
\end{align*}
Also, by $r\log x \gg 1$, the second $o$-term can be included in the first $o$-term.
Therefore, (15) follows. Similarly, (16) follows from Lemma 2.3(8) and (17) follows from (13).
\end{proof}
We use $p_1$, $p_2$, $\ldots$ , $p_v$ to denote prime numbers. We define the following multiple sums for $2\leq v\leq x$:
$$
\mathfrak{T}_{v,z}(x):=\sum_{p_1p_2\cdots p_v\leq x} \frac{\tau_z(p_1-1)\tau_z(p_2-1)\cdots\tau_z(p_v-1)}{p_1p_2\cdots p_v},
$$
and for $\mathbf{u}=(u_1,\ldots, u_v)$ with $1\leq u_i\leq x$,
$$
\mathfrak{T}_{\mathbf{u},v,z}(x):=\sum_{\substack{{p_1p_2\cdots p_v\leq x}\\{\forall_i, \ p_i \equiv 1 \ \mathrm{ mod } \ u_i}}} \frac{\tau_z(p_1-1)\tau_z(p_2-1)\cdots\tau_z(p_v-1)}{p_1p_2\cdots p_v},
$$
Define $\mathbb{T}_v:=\{(t_1,\ldots, t_v) : \forall_i, \ t_i\in [0,1], \ t_1+\cdots + t_v \leq 1 \}$. We adopt the idea from Gauss' Circle Problem. Recall that $r=(v^{\frac 32}\log v)^{-1}$. Consider a covering of $\mathbb{T}_v$ by $v$-cubes of side-length $r$ of the form:

Let $s_1, \ldots, s_v$ be nonnegative integers, let
$$
B_{s_1, \ldots, s_v}:=\{(t_1,\ldots , t_v) : \forall_i, \ r s_i\leq t_i< r (s_i+1)  \}.$$
Let $M_v$ be the set of those $v$-cubes lying completely inside $\mathbb{T}_v$. Then the sum $\mathfrak{T}_{v,z}(x)$ is over the primes satisfying:
$$
\left(\frac{\log p_1}{\log x}, \ldots, \frac{\log p_v}{\log x}\right)\in \mathbb{T}_v.$$
Instead of the whole $\mathbb{T}_v$, we consider the contribution of the sum over primes satisfying:
$$
\left(\frac{\log p_1}{\log x}, \ldots, \frac{\log p_v}{\log x}\right)\in \cup M_v,
$$
which come from the $v$-cubes lying completely inside $\mathbb{T}_v$. We define
$$
 {\mathfrak{S}_{v,z}}(x):=\sum_{\left(\frac{\log p_1}{\log x}, \ldots, \frac{\log p_v}{\log x}\right)\in \cup M_v} \frac{\tau_z(p_1-1)\tau_z(p_2-1)\cdots\tau_z(p_v-1)}{p_1p_2\cdots p_v},
$$
and similarly for $\mathbf{u}=(u_1,\cdots, u_v)$ with $1\leq u_i\leq x$,
$$
 {\mathfrak{S}_{\mathbf{u},v,z}}(x):=\sum_{\substack{{\left(\frac{\log p_1}{\log x}, \ldots, \frac{\log p_v}{\log x}\right)\in \cup M_v}\\{\forall_i, \ p_i \equiv 1 \ \mathrm{ mod } \ u_i}}} \frac{\tau_z(p_1-1)\tau_z(p_2-1)\cdots\tau_z(p_v-1)}{p_1p_2\cdots p_v},
$$
Let $v=\left\lfloor c \sqrt{\frac{\log x}{\log_2 x}}\right\rfloor$ for some positive constant $c$ to be determined. Then $v$ satisfies the conditions in Corollary 2.2. Then we have:
\begin{lemma}
Let $\log^{\frac1A} x< z \leq \log^A x$, then as $x\rightarrow\infty$,
\begin{equation}
 {\mathfrak{S}_{v,z}}(x)=\frac1{v!}S_z(x)^v(1+o(1))^v.
\end{equation}
For $\mathbf{u}=(u_1, u_2, 1, \ldots, 1)$ with $1\leq u_i\leq x$,
\begin{equation}
 {\mathfrak{S}_{\mathbf{u},v,z}}(x) \ll \frac{\tau(u_1)\tau(u_2)}{\phi(u_1)\phi(u_2)} {\mathfrak{S}_{v,z}}(x)\log^k z,
\end{equation}
where $0\leq k\leq 2$ is the number of $u_i$'s that are not $1$.

Assume that each $u_i$, $i=1, 2$  is a  positive integer with $p(u_i)>z$, $u_i< (\log x)^{A_1}$ and $\tau(u_i)<A_1$. Then as $x\rightarrow\infty$, we have
\begin{equation}
 {\mathfrak{S}_{\mathbf{u},v,z}}(x)=\frac{\tau(u_1)\tau(u_2)}{u_1u_2} {\mathfrak{S}_{v,z}}(x)\left(1+o(1)\right).\end{equation}
\end{lemma}
\begin{proof}[Proof of (18)]
It is clear that
$$
\mathrm{vol}\left((1-r\sqrt{v})\mathbb{T}_v\right)\leq |M_v| \mathrm{vol} (B_{0, \ldots, 0}) \leq \mathrm{vol} (\mathbb{T}_v).
$$
We have $\mathrm{vol} (\mathbb{T}_v)=\frac1{v!}$, $\mathrm{vol} (B_{0, \ldots, 0})=r^v$, and $\mathrm{vol} \left((1- r\sqrt{v})\mathbb{T}_v\right)=\frac1{v!}\left(1- r\sqrt{v}\right)^v$. Also, recall that  $r:={(v^{\frac32}\log v)}^{-1}$. Then,
$$
\frac{\frac1{v!}\left(1-\frac1{v\log v}\right)^v}{(v^{\frac32}\log v)^{-v}}\leq |M_v|\leq \frac{\frac1{v!}}{(v^{\frac32}\log v)^{-v}}.
$$
On the other hand, by Corollary 2.2(15), the contribution of each $v$-cube $[\alpha_1,\beta_1]\times \cdots \times [\alpha_v, \beta_v]\subseteq [0,1]^v$ of side-length $r$ to the sum is
$$
\sum_{\forall_i, \ \alpha_i\leq \frac{\log p_i}{\log x}<\beta_i} \frac{\tau_z(p_1-1)\tau_z(p_2-1)\cdots\tau_z(p_v-1)}{p_1p_2\cdots p_v}=\left(\prod_{i=1}^v (\beta_i-\alpha_i) \right) S_z(x)^v ( 1+o(1))^v=r^v S_z(x)^v(1+o(1))^v.
$$
Combining this with the bounds for $|M_v|$, we obtain the result.
\end{proof}
\begin{proof}[Proof of (19), (20)]
Let $v$ and $r$ be as defined in Corollary 2.2. We write (15) and (17) in the form of
\begin{equation}
\sum_{\alpha \leq \frac{\log p}{\log x} < \beta} \frac{\tau_z(p-1)}p=(\beta-\alpha)S_z(x)\left(1+f_{\alpha,\beta}(x)\right),
\end{equation}
and
\begin{equation}
\sum_{\substack{{\alpha \leq \frac{\log p}{\log x} < \beta} \\{p\equiv 1\ \mathrm{ mod } \ u}}} \frac{\tau_z(p-1)}p=(\beta-\alpha)\frac{\tau(u)}uS_z(x)\left(1+g_{\alpha, \beta}(x)\right).
\end{equation}
We note that there is a function $f(x)=o(1)$ such that uniformly for $0\leq \alpha\leq \beta\leq 1$ and $\beta-\alpha\geq r$,
$$
\textrm{max}(|f_{\alpha,\beta}(x)|, |g_{\alpha,\beta}(x)|)\leq f(x).$$
Then we can write
\begin{align*}
\sum_{\substack{{\alpha \leq \frac{\log p}{\log x} < \beta} \\{p\equiv 1\ \mathrm{ mod } \ u}}} \frac{\tau_z(p-1)}p&=(\beta-\alpha)\frac{\tau(u)}uS_z(x)\left(1+g_{\alpha, \beta}(x)\right)\\
&=\frac{\tau(u)}u\sum_{\alpha \leq \frac{\log p}{\log x} < \beta} \frac{\tau_z(p-1)}p \left(\frac{1+g_{\alpha,\beta}(x)}{1+f_{\alpha,\beta}(x)}\right)\\
&=\frac{\tau(u)}u\sum_{\alpha \leq \frac{\log p}{\log x} < \beta} \frac{\tau_z(p-1)}p \left(1+O(f(x))\right).
\end{align*}
Consider any $v$-cube $[\alpha_1,\beta_1]\times \cdots \times [\alpha_v, \beta_v]\subseteq [0,1]^v$ of side-length $r$. Then by the above observation,
\begin{align*}
\sum_{\substack{{\forall_i, \ \alpha_i\leq \frac{\log p_i}{\log x}<\beta_i}\\{p_i \equiv 1 \ \mathrm{ mod } \ u_i \ \mathrm{ for }\  i=1, \ 2}}} & \frac{\tau_z(p_1-1)\tau_z(p_2-1)\cdots\tau_z(p_v-1)}{p_1p_2\cdots p_v}\\
&=\frac{\tau(u_1)\tau(u_2)}{u_1u_2}\sum_{\forall_i, \ \alpha_i\leq \frac{\log p_i}{\log x}<\beta_i} \frac{\tau_z(p_1-1)\tau_z(p_2-1)\cdots\tau_z(p_v-1)}{p_1p_2\cdots p_v} (1+O(f(x)))^2.
\end{align*}
This proves (20). For the proof of (19), we use instead
\begin{align*}
\sum_{\substack{{\alpha \leq \frac{\log p}{\log x} < \beta} \\{p\equiv 1\ \mathrm{ mod } \ u}}} \frac{\tau_z(p-1)}p&=\frac{R_{u,z}(t)}t|_{x^{\alpha}}^{x^{\beta}}+\int_{x^{\alpha}}^{x^{\beta}} \frac{R_{u,z}(t)}{t^2}dt\\
&\ll \frac{\tau(u)}{\phi(u)}\left((\beta-\alpha)\log x + O(1)\right)
\ll \frac{\tau(u)}{\phi(u)}(\beta-\alpha)\log x \\
&\ll \frac{\tau(u)}{\phi(u)}(\beta-\alpha) S_z(x) \log z\ll \frac{\tau(u)}{\phi(u)} \sum_{\alpha\leq \frac{\log p}{\log x}<\beta} \frac{\tau_z(p-1)}p \log z,\end{align*}
which follows from Lemma 2.3(7).
\end{proof}
We impose some restrictions on the primes $p_1$, $\ldots$ , $p_v$:\\

R1. $p_1, \ldots , p_v$ are distinct.

R2. For each $i$, $q^2\nmid p_i-1$ for any prime $q>z$.

R3. $q^2\nmid \phi(p_1\cdots p_v)$ for any prime $q>z^2$.\\

Recall that we chose
$$v=\left\lfloor c \sqrt{\frac{\log x}{\log_2 x}} \right\rfloor$$
for some positive constant $c$ to be determined.
Let $ {\mathfrak{S}_{v,z}}^{(1)}(x)$ be the contribution of primes to $ {\mathfrak{S}_{v,z}}(x)$ not satisfying R1. Note that if R1 is not satisfied, then some primes among $p_1$, $\ldots$ , $p_v$ are repeated. Then by Lemma 2.6(18),
\begin{align*}
 {\mathfrak{S}_{v,z}}^{(1)}(x)&\ll \binom v2 \left( \sum_{z<p\leq x} \frac{\tau_z(p-1)^2}{p^2} \right) {\mathfrak{S}_{v-2, z}}(x)\\
&\ll v^2 \frac{\log^3 z}{z} \frac{v(v-1)}{S_z(x)^2}  {\mathfrak{S}_{v,z}}(x)\\
&\ll \frac{v^4\log^5 z}{z\log^2 x}  {\mathfrak{S}_{v,z}}(x)\ll \frac{\log^3 z}z  {\mathfrak{S}_{v,z}}(x).
\end{align*}
Let $ {\mathfrak{S}_{v,z}}^{(2)}(x)$ be the contribution of primes to $ {\mathfrak{S}_{v,z}}(x)$ not satisfying R2. Note that if R2 is not satisfied, then $q^2|p_i-1$ for some primes $p_i$ and $q>z$. Let $\mathbf{u}_{q^2}:=(q^2, 1, \ldots , 1)$. Suppose that $q^2|p_i-1$ for some $p_i$ and $q>z^2$. Then the contribution of those primes to $ {\mathfrak{S}_{v,z}}^{(2)}(x)$ is by (19),
\begin{align*}
\ll\sum_{q>z^2} \binom v1  {\mathfrak{S}_{\mathbf{u}_{q^2}, v,z}}(x)&\ll \sum_{q>z^2}\frac v{\phi(q^2)} {\mathfrak{S}_{v,z}}(x) \log z\ll \sum_{q>z^2} \frac v{q^2}  {\mathfrak{S}_{v,z}}(x)\log z \ll \frac{v}{z^2} {\mathfrak{S}_{v,z}}(x).
\end{align*}
Suppose that $q^2|p_i-1$ for some $p_i$ and $z<q\leq z^2$, then we have by (20),
\begin{align*}
\ll\sum_{z<q\leq z^2} \binom v1  {\mathfrak{S}_{\mathbf{u}_{q^2}, v,z}}(x)&\ll \sum_{z<q\leq z^2}\frac v{q^2} {\mathfrak{S}_{v,z}}(x)\ll \frac v{z\log z}  {\mathfrak{S}_{v,z}}(x).
\end{align*}
Thus, we have
$$
 {\mathfrak{S}_{v,z}}^{(2)}(x)\ll \frac v{z\log z}  {\mathfrak{S}_{v,z}}(x).
$$
Let $ {\mathfrak{S}_{v,z}}^{(3)}(x)$ be the contribution of primes to $ {\mathfrak{S}_{v,z}}(x)$ satisfying R1 and R2, but not satisfying R3. Note that if R1, R2 are satisfied and R3 is not satisfied, then there are at least two distinct primes $p_i$, $p_j$ such that $q|p_i-1$ and $q|p_j-1$. Let $\mathbf{u}_{q,q}:=(q,q,1, \ldots , 1)$. Suppose first that this happens with $q>z^4$. Then by (19), the contribution is
\begin{align*}
&\ll \sum_{q>z^4} \binom v2  {\mathfrak{S}_{\mathbf{u}_{q,q}, v,z}}(x)\ll \sum_{q>z^4} \frac{v^2}{\phi(q)^2} {\mathfrak{S}_{v,z}}(x) \log^2 z \ll\frac{v^2 \log z}{z^4}  {\mathfrak{S}_{v,z}}(x).
\end{align*}
Suppose that this happens with $z^2<q\leq z^4$. Then by (20), the contribution is
\begin{align*}
&\ll \sum_{z^2<q\leq z^4} \binom v2  {\mathfrak{S}_{\mathbf{u}_{q,q}, v,z}}(x)\ll \sum_{z^2<q\leq z^4} \frac{v^2}{q^2}  {\mathfrak{S}_{v,z}}(x) \ll \frac{v^2}{z^2\log z}  {\mathfrak{S}_{v,z}}(x).
\end{align*}
Thus, we have
$$
 {\mathfrak{S}_{v,z}}^{(3)}(x)\ll \frac{v^2}{z^2\log z}  {\mathfrak{S}_{v,z}}(x).
$$
We write $ {\mathfrak{S}_{v,z}}^{(0)}(x)$ to denote the contribution of those primes to $ {\mathfrak{S}_{v,z}}(x)$ satisfying all three restrictions R1, R2, and R3. By the above estimates, we have
\begin{align*}
 {\mathfrak{S}_{v,z}}^{(0)}(x)&\geq {\mathfrak{S}_{v,z}}(x)- {\mathfrak{S}_{v,z}}^{(1)}(x)- {\mathfrak{S}_{v,z}}^{(2)}(x)- {\mathfrak{S}_{v,z}}^{(3)}(x) \\
&=  {\mathfrak{S}_{v,z}}(x)\left(1+O\left(\frac{\log^3 z}z\right)+O\left(\frac v{z\log z}\right)+O\left(\frac{v^2}{z^2\log z}\right) \right).\end{align*}
Therefore,
\begin{equation}
 {\mathfrak{S}_{v,z}}^{(0)}(x)= {\mathfrak{S}_{v,z}}(x)\left(1+O\left(\frac{\log^3 z}z\right)+O\left(\frac v{z\log z}\right)+O\left(\frac{v^2}{z^2\log z}\right) \right).
\end{equation}
\section{Proof of Theorem 1.1}
We set
$$
v=v(x):=\left\lfloor c \sqrt{\frac{\log x}{\log_2 x}} \right\rfloor, \ \  z=z(x):=\sqrt{\log x},$$
$$
y:=\exp\left(\sqrt{\log x}\right)$$
with a positive constant $c$ to be determined.

Consider a subset $Q_z(x)$ of primes defined by:
$$
Q=Q_z(x):=\{p : p \leq x, \ q^2 \nmid p-1 \ \textrm{ for any prime } q >z \}.$$
We define $\mathcal{N}$, $\mathcal{M}$ by:
$$
\mathcal{N}=\mathcal{N}_v(x):=\{n\leq x : n \ \textrm{ is square-free, } p|n \ \Rightarrow p\in Q, \ w(n)=v  \},
$$
$$
\mathcal{M}=\mathcal{M}_v(x):=\{n\leq x : n\in \mathcal{N},  \ q^2 \nmid \phi(n) \ \textrm{ for any prime } q>z^2 \}.$$
We write
$$
V_{\mathcal{M}}(x):=\sum_{n\in\mathcal{M}} \frac{\tau_z(\lambda(n))}n, \ \ \tau''_z(n):=\prod_{p|n} \tau_z(p-1).$$
We also write
$$
W_{\mathcal{M}} :=\sum_{n\in\mathcal{M}} \frac{\tau''_z(n)}n, \ \ W_{\mathcal{M}}':=\sum_{n\in\mathcal{M}} \frac{\tau''_{z^2}(n)}n.$$

By (23), the contribution of those primes satisfying R1, R2, and R3 to $ {\mathfrak{S}_{v,z}}(x)$, which we wrote as $ {\mathfrak{S}_{v,z}}^{(0)}(x)$ satisfies
\begin{align*}
 {\mathfrak{S}_{v,z}}^{(0)}(x)&=  {\mathfrak{S}_{v,z}}(x)\left(1+O\left(\frac{\log^3 z}{z}\right)+O\left(\frac v{z\log z}\right)+O\left(\frac{v^2}{z^2\log z}\right) \right).\\
&= {\mathfrak{S}_{v,z}}(x)\left(1+O\left(\frac1{\log_2 x}\right)\right).\end{align*}
Then by Lemma 2.6(18) and Stirling's formula,
$$
W_{\mathcal{M}}\geq \frac1{v!} {\mathfrak{S}_{v,z}}^{(0)}(x)\asymp \frac1{v} \left(\frac ev\right)^{2v} \left(c_1\frac{\log x}{\log z}\right)^v(1+o(1))^{v}
$$
Thus,
$$
W_{\mathcal{M}}\gg \exp\left( \sqrt{\frac{\log x}{\log_2 x}} \left( 2c + c \log c_1 - 2 c\log c + c \log 2+o(1)\right) \right).
$$
Maximizing $2c + c \log c_1 - 2 c\log c + c \log 2$ by the first derivative, we have $c=\sqrt 2 e^{-\gamma/2}$, hence
$$
W_{\mathcal{M}}\gg \exp\left(2\sqrt 2 e^{-\frac{\gamma}2} \sqrt{\frac{\log x}{\log_2 x}} (1+o(1)) \right).
$$
For $W_{\mathcal{M}}'$, we have by (23), the contribution of those primes satisfying R1, R2, and R3 to $ {\mathfrak{S}_{v,z^2}}(x)$, say $ {\mathfrak{S}_{v,z^2}}^{(0')}(x)$ satisfies
\begin{align*}
 {\mathfrak{S}_{v,z^2}}^{(0')}(x)&=  {\mathfrak{S}_{v,z^2}}(x)\left(1+O\left(\frac{\log^3 z}{z^2}\right)+O\left(\frac v{z\log z}\right)+O\left(\frac{v^2}{z^2\log z}\right) \right).\\
&= {\mathfrak{S}_{v,z^2}}(x)\left(1+O\left(\frac1{\log_2 x}\right)\right).\end{align*}
Then by Lemma 2.6(18) and Stirling's formula, as $x\rightarrow\infty$,
$$
W_{\mathcal{M}}'\geq \frac1{v!} {\mathfrak{S}_{v,z^2}}^{(0')}(x)\asymp \frac1{v} \left(\frac ev\right)^{2v} \left(c_1\frac{\log x}{\log z^2}\right)^v(1+o(1))^{v}
$$

Thus,
$$
W_{\mathcal{M}}'\gg \exp\left( \sqrt{\frac{\log x}{\log_2 x}} \left( 2c + c \log c_1 - 2 c\log c  +o(1)\right) \right).
$$
Maximizing $2c + c \log c_1 - 2 c\log c$ by the first derivative, we have $c=  e^{-\gamma/2}$, hence as $x\rightarrow\infty$,
$$
W_{\mathcal{M}}'\gg \exp\left(2 e^{-\frac{\gamma}2} \sqrt{\frac{\log x}{\log_2 x}} (1+o(1)) \right).
$$
Therefore, we have just proved the lower bounds of the following:
\begin{theorem}For $z=\sqrt{\log x}$, as $x\rightarrow\infty$,
\begin{equation}\sum_{n\leq x}\mu^2 (n) \frac{\tau''_z(n)}n =\exp\left(2\sqrt 2 e^{-\frac{\gamma}2}\sqrt{\frac{\log x}{\log_2 x}}(1+o(1))  \right),\end{equation}
and
\begin{equation}\sum_{n\leq x} \mu^2(n)\frac{\tau''_{z^2}(n)}n =\exp\left(2e^{-\frac{\gamma}2}\sqrt{\frac{\log x}{\log_2 x}}(1+o(1)) \right).\end{equation}
\end{theorem}
Note that the upper bounds follow from Rankin's method as in ~\cite[Theorem 1]{LP}. \\

We proceed the similar argument as in ~\cite{LP}. Let $\mathcal{M}=\mathcal{M}_v(x)$ be as above with the choice $c=e^{-\gamma/2}$. Now, for $n\in \mathcal{M}$, we have
$$
\tau_z(\phi(n))=\tau_{z,z^2}(\phi(n)) \tau_{z^2} (\phi(n)) \geq  \tau_{z^2}(\phi(n))= \tau''_{z^2}(n),$$
$$
\tau_z(\lambda(n))=\tau_{z,z^2}(\lambda(n)) \tau_{z^2} (\lambda(n)) \geq  \tau_{z^2}(\lambda(n)) =\tau''_{z^2}(n).$$

Then as $x\rightarrow\infty$,
$$
V_{\mathcal{M}}(x)\geq W_{\mathcal{M}}'\gg \exp\left( 2e^{-\frac{\gamma}2} \sqrt{\frac{\log x}{\log_2 x}}(1+o(1))\right).$$
The argument proceeds as in ~\cite{LP}. Let $\mathcal{M'}$ be defined by
$$
\mathcal{M'}:=\left\{ np : n\in \mathcal{M}_v(xy^{-1}), \ \ p \ \textrm{is a prime}, \ \  p \leq \frac xn \right\}.$$
For those $n'=np\in \mathcal{M'}$, we have
$$
\tau(\lambda(np))\geq \tau(\lambda(n))\geq \tau_z(\lambda(n)),$$
and a given $n'\in \mathcal{M'}$ has at most $v+1$ decompositions of the form $n'=np$ with $n\in \mathcal{M}_v(xy^{-1})$, $ p\leq \frac xn$.

Since $n\leq x y^{-1}$ for $n\in\mathcal{M}_v(xy^{-1})$, the number of $p$ in $ p\leq \frac xn$ is
$$
\pi\left(\frac xn\right)  \gg \frac{x}{n\log x}.$$
Note that $\log y = \sqrt{\log x}=o(\log x)$.
This gives
$$
V_{\mathcal{M}}(xy^{-1})\gg \exp\left( 2e^{-\frac{\gamma}2} \sqrt{\frac{\log x}{\log_2 x}}(1+o(1))\right).$$
Then
$$
\sum_{n\leq x}\tau(\lambda(n))\geq \sum_{n\in \mathcal{M'}}\tau(\lambda(n)) \gg V_{\mathcal{M}}(xy^{-1}) \frac x{v\log x}\gg x\exp\left( 2e^{-\frac{\gamma}2} \sqrt{\frac{\log x}{\log_2 x}}(1+o(1))\right).
$$
This completes the proof of Theorem 1.1.\\

\nid \bf Remarks. \rm \\

\nid 1. In the proof of Theorem 1.1, we dropped $\tau_{z,z^2}(\phi(n))$. This is where a prime $z<q\leq z^2$ can divide multiple $p_i-1$ for $i=1, 2, \cdots, v$, and that is the main difficulty in obtaining more precise formulas for $\sum_{n\leq x}\tau(\phi(n))$ and $\sum_{n\leq x}\tau(\lambda(n))$.\\

\nid 2. We will see a heuristic argument suggesting that as $x\rightarrow\infty$,
$$
\sum_{n\leq x} \tau(\lambda(n))=x\exp\left( 2\sqrt 2 e^{-\frac{\gamma}2}\sqrt{\frac{\log x}{\log_2 x}}(1+o(1)) \right),
$$
and hence,
$$
\sum_{n\leq x} \tau(\phi(n))=x\exp\left( 2\sqrt 2 e^{-\frac{\gamma}2}\sqrt{\frac{\log x}{\log_2 x}}(1+o(1)) \right).
$$
However, we have
$$
\sum_{n\leq x}\tau(\lambda(n)) = o\left(\sum_{n\leq x}\tau(\phi(n))\right).$$
We will prove this in the following section. The prime $2$ plays a crucial role in the proof of Theorem 1.2.
\section{Proof of Theorem 1.2}
We put $k$ and $w$ as in ~\cite{LP}:
$$
k=\lfloor A\log_2 x \rfloor, \ \ \omega=\left\lfloor \frac{\sqrt{\log x}}{\log_2^2 x}\right\rfloor.
$$
Here, $A$ is a positive constant to be determined.
Also, define $\mathcal{E}_1(x)$, $\mathcal{E}_2(x)$ and $\mathcal{E}_3(x)$ in the same way:
$$
\mathcal{E}_1(x):= \{n\leq x : 2^k | n \ \textrm{or} \ \textrm{there is a prime $p|n$ with} \ p\equiv 1 \ \mathrm{mod} \ 2^k   \},
$$
$$
\mathcal{E}_2(x):=\{ n\leq x : \omega(n)\leq \omega \},
$$
and
$$
\mathcal{E}_3(x):=\{ n\leq x \} - \left(\mathcal{E}_1(x)\cup \mathcal{E}_2(x)\right).
$$
We need the following lemma.
\begin{lemma}
For any $2\leq y\leq x$, we have
$$
\sum_{n\leq \frac xy} \frac{\tau(\phi(n))}n \ll \frac{\log^5 x}x \sum_{n\leq x} \tau(\phi(n)).
$$
\end{lemma}
\begin{proof}
As in the proof of ~\cite[Theorem 1]{LP}, we use the square-free kernel $k=k(n)$ (if a prime $p$ divides $n$, then $p|k$, and $k$ is a square-free positive integer which divides $n$) and the factorization $n=mk$ to rewrite the sum as
\begin{align*}
\sum_{n\leq \frac xy} \frac{\tau(\phi(n))}n &\leq \sum_{k\leq \frac xy}\mu^2(k) \sum_{m\leq \frac x{ky}} \frac{\tau(m)\tau(\phi(k))}{mk}\\
&\ll \sum_{k\leq \frac xy}\mu^2(k) \frac{\tau(\phi(k))}k \log^2 x.
\end{align*}
Note that we have uniformly $w(n)\ll \log x$. Find $v$ such that
$$
\sum_{\substack{{k\leq \frac xy}\\{\omega(k)=v}}}\mu^2(k) \frac{\tau(\phi(k))}k $$
is maximal. Then we have
$$
\sum_{k\leq \frac xy}\mu^2(k) \frac{\tau(\phi(k))}k \ll \log x \sum_{\substack{{k\leq \frac xy}\\{\omega(k)=v}}}\mu^2(k) \frac{\tau(\phi(k))}k.
$$
We adopt an idea from the proof of Theorem 1.1.
Let $\mathcal{M}=\mathcal{M}_v(xy^{-1})$ be the set of square-free numbers $k\leq xy^{-1}$ with $\omega(k)=v$.
Define
$$\mathcal{M}':= \left\{ kp : k\in \mathcal{M}_v(xy^{-1}), \ \ p \ \textrm{is a prime}, \ \ p \leq \frac xk \right\}.$$
For those $n'=kp\in \mathcal{M'}$ with $k\in \mathcal{M}$, we have
$$
\tau(\phi(kp))\geq \tau(\phi(k)),$$
and  any given $n'\in \mathcal{M'}$ has at most $v+1$ decompositions of the form $n'=kp$ with $k\in \mathcal{M}$, $p\leq \frac xk$.

Since  the number of $p$ satisfying $p\leq \frac xk$ is
$$
\pi\left(\frac xk\right)  \gg \frac{x}{k\log x},$$
it follows that
$$
\sum_{n\leq x}\tau(\phi(n))\geq \sum_{n\in \mathcal{M'}}\tau(\phi(n)) \gg \sum_{\substack{{k\leq \frac xy}\\{\omega(k)=v}}}\mu^2(k) \frac{\tau(\phi(k))}k\frac x{v\log x}.
$$
Since $v\ll \log x$, we have
$$
\sum_{\substack{{k\leq \frac xy}\\{w(k)=v}}}\mu^2(k) \frac{\tau(\phi(k))}k \ll \frac{\log^2 x}x \sum_{n\leq x}\tau(\phi(n)).$$
This gives
$$
\sum_{k\leq \frac xy}\mu^2(k) \frac{\tau(\phi(k))}k \ll \frac{\log^3 x}x \sum_{n\leq x}\tau(\phi(n)).
$$
Then the result follows.
\end{proof}
For $n\in \mathcal{E}_1(x)$, we have by Lemma 2.3 and Lemma 4.1,
\begin{align*}
\sum_{n\in\mathcal{E}_1(x)} \tau(\lambda(n))&\leq x\sum_{n\in\mathcal{E}_1(x)}\frac{\tau(\phi(n))}n\\
&\leq x \frac{\tau(2^k)}{2^k}\sum_{m\leq \frac x{2^k}}\frac{\tau(\phi(m))}m+ x \sum_{\substack{{p\leq x}\\{p \equiv 1 \ \mathrm{mod} \ 2^k}}}\frac{\tau(p-1)}p\sum_{m\leq \frac xp} \frac{\tau(\phi(m))}m \\
&\ll \log^5 x \left(\frac{\tau(2^k)}{\phi(2^k)} \log x \sum_{n\leq x}\tau(\phi(n))\right)\\
&\ll \log^6 x \frac{A\log_2 x}{\log^{A\log 2} x} \sum_{n\leq x}\tau(\phi(n)).
\end{align*}
If we take $A\log 2 > 7$, then we obtain that
$$
\sum_{n\in \mathcal{E}_1(x)} \tau(\lambda(n)) = o\left( \sum_{n\leq x} \tau(\phi(n))\right).
$$
For $n\in \mathcal{E}_2(x)$, we use the square-free kernel $k=k(n)$ and the factorization $n=mk$ as before,
\begin{align*}
\sum_{n\in\mathcal{E}_2(x)}\tau(\lambda(n))&\leq \sum_{n\in\mathcal{E}_2(x)}\tau(\phi(n))\\
&\ll\sum_{\substack{{k\leq x}\\{\omega(k)\leq \omega}}}\mu^2(k)\sum_{m\leq \frac xk} \tau(m)\tau(\phi(k))\\
&\ll\sum_{\substack{{k\leq x}\\{\omega(k)\leq \omega}}}\mu^2(k) \frac xk (\log x) \tau(\phi(k))\\
&\ll x\omega\log x\left(\sum_{p\leq x}\frac{\tau(p-1)}p\right)^{\omega}\\
&\ll x (\log x)^{\frac 32} (C\log x)^{\omega} \ll x  \exp\left( 2\frac{\sqrt{\log x}}{\log_2 x}\right).
\end{align*}
Thus, by Theorem 1.1,
$$
\sum_{n\in\mathcal{E}_2(x)}\tau(\lambda(n))=o\left(\sum_{n\leq x}\tau(\phi(n))\right).
$$
For $n\in \mathcal{E}_3(x)$, we follow the method of ~\cite{LP}. We have
\begin{align*}
\frac{\tau(\phi(n))}{\tau(\lambda(n))}\geq \frac {\omega}k \gg \frac{\sqrt{\log x}}{\log_2^3 x}.
\end{align*}
Then
\begin{align*}
\sum_{n\in\mathcal{E}_3(x)} \tau(\lambda(n))\ll \frac{\log_2^3 x}{\sqrt{\log x}}\sum_{n\in\mathcal{E}_3(x)} \tau(\phi(n))\leq \frac{\log_2^3 x}{\sqrt{\log x}} \sum_{n\leq x}\tau(\phi(n)).
\end{align*}
Therefore, putting these together, we have
$$
\sum_{n\leq x}\tau(\lambda(n))\ll \frac{\log_2^3 x}{\sqrt{\log x}}\sum_{n\leq x}\tau(\phi(n)),
$$
and Theorem 1.2 follows.
\section{Heuristics}
Recall that $\tau_z(\lambda(n))=\tau_{z,z^2}(\lambda(n)) \tau_{z^2} (\lambda(n))$. Let $\mathcal{M}$ be the set defined in Section 3 with the choice of $v=\left\lfloor \sqrt 2 e^{-\gamma/2} \sqrt{\frac{\log x}{\log_2 x}}\right\rfloor$. As in Section 3, we have $\tau_{z^2}(\lambda(n))=\tau''_{z^2}(n)$ for $n\in\mathcal{M}$. It is important to note that $q^2\nmid p_i-1$ for any primes $p_i|n$ and $q>z$. Also, we have $q^2\nmid \phi(n)$ for $q>z^2$. Thus, it is enough to focus on the sum $V_{\mathcal{M}}(x)$. If we could prove that $V_{\mathcal{M}}(x)=\sum_{n\in\mathcal{M}}\frac{\tau_z(\lambda(n))}n\gg \exp\left( 2\sqrt 2 e^{-\frac{\gamma}2} \sqrt{\frac{\log x}{\log_2 x}}(1+o(1))\right)$, then the same argument as in Theorem 1.1 would allow $\sum_{n\leq x}\tau(\lambda(n))\gg x \exp\left( 2\sqrt 2 e^{-\frac{\gamma}2} \sqrt{\frac{\log x}{\log_2 x}}(1+o(1))\right)$. We need the contribution of $\tau_{z,z^2}(\lambda(n))$ over $n\in\mathcal{M}$. Let $\mathfrak{S}_{v,z}(x)$ be the sum defined in Section 2, and define
$$
\mathfrak{U}_{v,z}(x):=
 \sum_{\left(\frac{\log p_1}{\log x}, \ldots , \frac{\log p_v}{\log x}\right)\in \cup M_v} \frac{\tau_{z,z^2}(\mathrm{lcm}(p_1-1,p_2-1, \ldots ,p_v-1))}{\tau_{z,z^2}(p_1-1)\tau_{z,z^2}(p_2-1)\cdots \tau_{z,z^2}(p_v-1)}\frac{\tau_z(p_1-1)\tau_z(p_2-1)\cdots\tau_z(p_v-1)}{p_1p_2\cdots p_v}.
$$
We have also defined in Section 2 that for $\mathbf{u}=(u_1,\ldots , u_v)$ with $1\leq u_i\leq x$,
$$
 {\mathfrak{S}_{\mathbf{u},v,z}}(x):=\sum_{\substack{{\left(\frac{\log p_1}{\log x}, \ldots , \frac{\log p_v}{\log x}\right)\in \cup M_v}\\{\forall_i, \ p_i \equiv 1 \ \mathrm{ mod } \ u_i}}} \frac{\tau_z(p_1-1)\tau_z(p_2-1)\cdots\tau_z(p_v-1)}{p_1p_2\cdots p_v},
$$
We need to extend Lemma 2.6 to cover all components of $\mathbf{u}$.
\begin{lemma}
Let $\log^{\frac1A} x< z \leq \log^A x$, then
for $\mathbf{u}=(u_1, u_2,  \ldots , u_v)$ with $1\leq u_i\leq x$,
\begin{equation}
 {\mathfrak{S}_{\mathbf{u},v,z}}(x) \ll \frac{\tau(u_1)\tau(u_2)\cdots \tau(u_v)}{\phi(u_1)\phi(u_2)\cdots \phi(u_v)} {\mathfrak{S}_{v,z}}(x)(1+o(1))^k\log^k z,
\end{equation}
where $0\leq k\leq v$ is the number of $u_i$'s that are not $1$.

Assume that each $u_i$,  $1\leq i\leq v$  is either $1$ or a  positive integer with $p(u_i)>z$, $u_i< (\log x)^{A_1}$ and $\tau(u_i)<A_1$. Then
\begin{equation}
 {\mathfrak{S}_{\mathbf{u},v,z}}(x)=\frac{\tau(u_1)\tau(u_2)\cdots\tau(u_v)}{u_1u_2\cdots u_v} {\mathfrak{S}_{v,z}}(x)\left(1+o(1)\right)^k,\end{equation}
 where $0\leq k\leq v$ is the number of $u_i$'s that are not $1$.
\end{lemma}
The same proof as in Lemma 2.6 applies with the need of considering all components of $\mathbf{u}$.

Fix a prime $z<q\leq z^2$. Consider the number $X_q$ of primes $p_1, \ldots , p_v$ such that $q$ divides $p_i-1$.
By Lemma 5.1, it is natural to model $X_q$ by a binomial distribution with parameters $v$ and $\frac2q$. In fact, Lemma 5.1 implies that
\begin{lemma}
For any $0\leq k\leq v$, as $x\rightarrow\infty$,
\begin{align*}
P(X_q=k):&=\frac1{\mathfrak{S}_{v,z}(x)}\sum_{\substack{{\left(\frac{\log p_1}{\log x}, \ldots , \frac{\log p_v}{\log x}\right)\in \cup M_v}\\{\textrm{Exactly } k \textrm{ primes }p_i\textrm{ satisfy }q|p_i-1}}} \frac{\tau_z(p_1-1)\tau_z(p_2-1)\cdots\tau_z(p_v-1)}{p_1p_2\cdots p_v}\\
&=\binom vk \left(\frac2q\right)^k \left(1-\frac2q\right)^{v-k}(1+o(1))^v.
\end{align*}
Here, the functions implied in $1+o(1)$ only depend  on $x$ and   do not depend on $k$.
\end{lemma}
Denote by $A_q$ the contribution of a power of $q$ in
$$
\frac{\tau_{z,z^2}(\mathrm{lcm}(p_1-1,p_2-1, \ldots ,p_v-1))}{\tau_{z,z^2}(p_1-1)\tau_{z,z^2}(p_2-1)\cdots \tau_{z,z^2}(p_v-1)}.
$$
Similarly, denote by $A_{q_1, \cdots, q_j}$ the contribution of powers of $q_1, \cdots, q_j$ in the above. Let
$$
B_{z,v}:=\frac{\tau_z(p_1-1)\tau_z(p_2-1)\cdots\tau_z(p_v-1)}{p_1p_2\cdots p_v}.
$$
We can combine the contributions of finite number of primes $q_1, \ldots , q_j$ in $(z,z^2]$. For these multiple primes, Lemma 5.2 becomes
\begin{lemma}
For any $0\leq k_1, \ldots , k_j \leq v$, as $x\rightarrow\infty$,
\begin{align*}
P(X_{q_1}=k_1, \ldots , X_{q_j}=k_j):&=\frac1{\mathfrak{S}_{v,z}(x)}\sum_{\substack{{\left(\frac{\log p_1}{\log x}, \ldots , \frac{\log p_v}{\log x}\right)\in \cup M_v}\\{\textrm{For each $s=1,\ldots , j,$}}\\{\textrm{ exactly }k_s \textrm{ primes }p_i\textrm{ satisfy }q_s|p_i-1}}} \frac{\tau_z(p_1-1)\tau_z(p_2-1)\cdots\tau_z(p_v-1)}{p_1p_2\cdots p_v}\\
&=\prod_{s\leq j} \binom v{k_s} \left(\frac2{q_s}\right)^{k_s} \left(1-\frac2{q_s}\right)^{v-k_s}(1+o(1))^v.
\end{align*}
Here, the functions implied in $1+o(1)$ only depend  on $j$, $x$ and they do not depend on $k_s$. 
\end{lemma}
This shows that the random variables $X_{q_i}$ behave similar as independent binomial distributions. For $z<q\leq z^2$, we have $A_q=\frac 2{2^k}$ for $k\geq 1$, and $A_q=1$ for $k=0$. Thus, the contribution of this prime $q$ is
$$\mathbf{E}[A_q]=\left(2\left( 1-\frac 1q\right)^v - \left(1-\frac 2q\right)^v\right) (1+o(1))^v.
$$
For distinct primes $q_1, \ldots , q_j$ in $(z,z^2]$, the contribution of these primes is
$$\mathbf{E}[A_{q_1,\ldots,q_j}]=\prod_{s\leq j} \left(2\left( 1-\frac 1{q_s}\right)^v - \left(1-\frac 2{q_s}\right)^v\right)(1+o(1))^v,$$
where the function implied in $1+o(1)$ only depends on $j$, $x$.

Then, we conjecture that the contribution of all primes in $z<q\leq z^2$ will be
\begin{conj}
As $x\rightarrow\infty$, we have
$$
\mathfrak{U}_{v,z}(x)=\prod_{z<q\leq z^2} \left(2\left( 1-\frac 1{q}\right)^v - \left(1-\frac 2{q}\right)^v\right)\mathfrak{S}_{v,z}(x)(1+o(1))^v.
$$
\end{conj}
It is clear that
$$
2\left( 1-\frac 1{q}\right)^v - \left(1-\frac 2{q}\right)^v= 1+o\left(\frac vq\right) .
$$
Thus, we have as $x\rightarrow\infty$,
$$
\prod_{z<q\leq z^2}\left(2\left( 1-\frac 1{q}\right)^v - \left(1-\frac 2{q}\right)^v\right)=(1+o(1))^v.$$
Therefore, we obtain the following heuristic result according to Conjecture 5.1.
\begin{conj}
As $x\rightarrow\infty$, we have
$$
\mathfrak{U}_{v,z}(x)=\mathfrak{S}_{v,z}(x)(1+o(1))^v.
$$
\end{conj}
Then Conjecture 1.1 follows from Lemma 2.6. \\

\nid \bf Remarks. \rm \\

We were unable to prove Conjecture 1.1. The main difficulty is due to the short range of $u$ in Corollary 2.1. Because of the range of $u$, we could not extend Lemma 5.3 to all primes in $z<q\leq z^2$.
              \flushleft


\begin{thebibliography}{99}
    \bibitem[BFI]{BFI} E. Bombieri, J. Friedlander, H. Iwaniec, \emph{Primes in arithmetic progressions to large moduli}, Acta Mathematica 156(1986), pp. 203--251.
    \bibitem[EP]{EP} P. Erd\H{o}s, C. Pomerance, \emph{On the Normal Order of Prime Factors of $\phi(n)$}, Rocky Mountain Journal of Mathematics, Volume 15, Number 2, Spring 1985.
    \bibitem[F]{F} D. Fiorilli, \emph{On a Theorem of Bombieri, Friedlander  and Iwaniec}, Canadian J. Math 64(2012), pp. 1019--1035
    \bibitem[HR]{HR} G. H. Hardy, S. Ramanujan, \emph{The Normal Number of Prime Factors of a Number $n$}, Quarterly Journal of Mathematics, Volume 48, pp. 76--92
    \bibitem[LP]{LP} F. Luca, C. Pomerance, \emph{On the Average Number of Divisors of the Euler Function}, Publ. Math. Debrecen, 70/1-2 (2007), pp 125--148.
    \bibitem[LP2]{LP2} F. Luca, C. Pomerance, \emph{Corrigendum: ‘On the Average Number of Divisors of the Euler Function’}, Publ. Math. Debrecen, 89 / 1-2 (2016)
    \bibitem[MV]{MV} H. Montgomery, R. Vaughan, \emph{Multiplicative Number Theory I. Classical Theory}, Cambridge University Press 2007
    \bibitem[Pe]{Pe} C. Pehlivan, \emph{Some Average Results Connected with Reductions of Groups of Rational Numbers}, Ph. D. Thesis (2015), Universit\`{a} Degli Studi Roma Tre.

      % bibliography references
 \end{thebibliography}
\end{document}